\documentclass[12pt,reqno]{amsart}
\usepackage{amsmath, amssymb, amsfonts, xcolor}
\usepackage{hyperref}
\numberwithin{equation}{section}
\usepackage{amsthm}
\newtheorem{thm}{Theorem}[section]
\newtheorem{example}{Example}
\newtheorem{prop}[thm]{Proposition}
\newtheorem{lem}[thm]{Lemma}

\newtheorem{dfn}[thm]{Definition}

\newtheorem{remark}[thm]{Remark}

\newtheorem{cor}[thm]{Corollary}

\usepackage[a4paper, top = 1.2in, bottom = 1.2in, left = 1.2in, right = 1.2in]{geometry}

\usepackage{enumitem}
\newlist{steps}{enumerate}{1}
\setlist[steps, 1]{label = Step \arabic*:}

\numberwithin{equation}{section}

\newcommand{\F}{\mathbb{F}}

\newcommand{\Q}{\mathbb{Q}}

\newcommand{\Z}{\mathbb{Z}}

\newcommand{\mcO}{\mathcal{O}}
\newcommand{\mcD}{\mathcal{D}}

\newcommand{\mfq}{\mathfrak{q}}
\newcommand{\mfP}{\mathfrak{P}}

\newcommand{\Gal}{\mathrm{Gal}}

\def\1{1\!\!1}

\title[GFE over cyclotomic $\Z_l$-extensions of $K$]{Generalized Fermat equation over cyclotomic $\Z_l$-extensions of totally real fields}
\author{Satyabrat Sahoo}

\address[S. Sahoo]{Shanghai Institute for Mathematics and Interdisciplinary Sciences (SIMIS), Shanghai 200433, China
	\& \newline
	Research Institute of Intelligent Complex Systems, Fudan University, Shanghai 200433, China}
\email{satyabrat.sahoo.94@gmail.com}

\date{\today}

\begin{document}
	\begin{abstract}
Let $K$ be a totally real field of odd degree $m$ in which $2$ is inert. Let $l\geq5$ be a prime totally ramified in $K$ with $\gcd(m,l(l-1)/2)=1$ and $2^{l-1}\not\equiv1\pmod{l^2}$. We prove the asymptotic Fermat's Last Theorem over every layer $K_{n,l}$, $n\geq1$, of the cyclotomic $\Z_l$-extension of $K$. Using $S$-unit bounds, we also give criteria for the generalized Fermat equation $Ax^p+By^p=Cz^p$ to have no asymptotic solutions in each layer $K_{n,l}$ when the coefficients $A,B,C$ are supported at the primes above $2$ or $2d$, with $d$ an odd prime. In particular, for every fixed integer $r\geq2$ other than $4$, the equation $x^p+y^p=2^rz^p$ has no asymptotic solutions in each $n$-th layer $K_{n,l}^3$. Thorne's modularity theorem makes the resulting bounds effective in each layer $\Q_{n,l}$ to which our criteria apply.
	\end{abstract}
	
	\maketitle
	
	\section{Introduction}
	Let $K$ be a totally real number field, and let $\mcO_K$ denote its ring of integers. The generalized Fermat equation of exponent $p$ over $K$ is the equation 
	\begin{equation}
		\label{Ax^p+By^p=Cz^p}
		Ax^p+By^p=Cz^p, \  A,B,C \in \mcO_K\setminus \{0\}.
	\end{equation}
The asymptotic Fermat's Last Theorem over $K$ is the statement that there is a bound $V_K$ (depending only on $K$) such that for all primes $p >V_K$, the only solutions to the equation $x^p+y^p=z^p$ with $x,y,z \in K$ are the trivial ones, i.e., $xyz=0$. If $V_K$ is effectively computable, we shall refer to this as the effective asymptotic Fermat's Last Theorem over $K$. 
	
	In \cite{FS15}, Freitas and Siksek proved an $S$-unit criterion for the asymptotic Fermat's Last Theorem over totally real fields. They also proved the effective asymptotic Fermat's Last Theorem over $\Q(\sqrt d)$ for a set of square-free integers $d\geq2$ of relative density $5/6$. Deconinck~\cite{D16} extended their method to~\eqref{Ax^p+By^p=Cz^p} with $ABC$ odd.
Kumar and the author~\cite{KS23 Diophantine1} studied asymptotic solutions of $x^p+2^ry^p=z^p$ for positive integers $r$. In~\cite{S26}, \cite{S26new}, the author studied the generalized Fermat equation~\eqref{Ax^p+By^p=Cz^p} with arbitrary nonzero coefficients using $S$-unit criteria and local restrictions on the coefficients.

	In \cite{FKS21}, Freitas, Kraus and Siksek proved that the effective asymptotic Fermat's Last Theorem holds over every layer 
	$\Q_{n,2}$ of the cyclotomic $\Z_2$-extension of $\Q$ (cf. \S\ref{preliminary section} for the $n$-th layer). 
	Freitas, Kraus and Siksek~\cite{FKS20b} also proved that if $K$ is totally real, $2$ is totally ramified in $K$ and the narrow class number $h_K^+$ of $K$ is odd, then the asymptotic Fermat's Last Theorem holds over any totally real $2$-extension of $K$ unramified away from $2$.
	
	Let $l\geq 5$ be a prime. In \cite{FKS20a}, the authors proved that if $l$ is non-Wieferich, i.e., $2^{l-1} \not\equiv 1 \pmod {l^2} $, then the effective asymptotic Fermat's Last Theorem holds over every layer $\Q_{n,l}$ of the cyclotomic $\Z_l$-extension of $\Q$. 
	In this article, we study the asymptotic solutions of the generalized Fermat equation \eqref{Ax^p+By^p=Cz^p}  over the cyclotomic $\Z_l$-extensions of totally real number fields $K$.

In Theorem~\ref{main result2}, we first establish the asymptotic Fermat's Last Theorem over each layer $K_{n,l}$, $n\geq1$, under the stated hypotheses on $K$ and $l$. We determine the possible valuations of the $S$-unit solutions at the unique prime $\mfP$ above $2$. With this calculation, \cite[Theorem 2.7]{S26} gives Theorem~\ref{main result2} for $A=B=C=1$. In Theorem~\ref{main result3}(1), we apply the same criterion to coefficients $A,B,C$ supported at the primes above $2$ and satisfying $\max\{v_\mfP(A),v_\mfP(BC)\}\leq4$ and $v_\mfP(ABC)\equiv0$ or $2\pmod3$. In part~(2), we use \cite[Theorem 1.6(1)]{S26new} when the valuations of $A,B,C$ at $\mfP$ are $m,m,m+k$, respectively, with $k\geq5$. Together, parts~(1) and~(2) show that $x^p+y^p=2^rz^p$ has no asymptotic solutions in each layer $K_{n,l}^3$, for every fixed integer $r\geq2$ with $r \neq 4$. For part~(3), we apply \cite[Theorem 1.6(2)]{S26new} when $\mcD_\mfP(A,B,C)\neq0$, and prove that for all sufficiently large primes $p$ there are no non-trivial solutions whose coordinate valuations at $\mfP$ are not all equal. In particular, this excludes every non-trivial primitive solution $(a, b, c)\in \mcO_{K_{n,l}}^3$ with $\mfP\mid abc$.

Finally, Theorem~\ref{thm extension two d} treats coefficients supported at primes above $2d$ in each fixed layer $K_{n,l}$ with odd narrow class number. The hypotheses on the prime $d$ ensure that it remains inert in the layer, where a descent gives the required $S$-unit bound. In part~(1), we apply \cite[Theorem 1.6(1)]{S26new} when $v_\mfP(A)=v_\mfP(B)=m$ and $v_\mfP(C)=m+k$ with $k\geq5$. In part~(2), we use \cite[Theorem 1.6(2)]{S26new} to treat unequal coordinate valuations. Taking $K=\Q$ gives Corollary~\ref{cor extension rational tower} for the rational cyclotomic layers. Thorne~\cite[Theorem 1]{Thorne19} proved that for any prime $l$ and any number field contained in the cyclotomic $\Z_l$-extension of $\Q$, every elliptic curve over that field is modular. Applying this theorem to $\Q_{n,l}$ gives the effective bounds in Corollary~\ref{cor extension rational tower}.

	\subsection{Main results}
	\label{notations section for x^r+y^r=dz^p}
	We first prove that the asymptotic Fermat's Last Theorem holds over every layer 
	$K_{n,l}$ of the cyclotomic $\Z_l$-extension of $K$.
	\begin{thm}
		\label{main result2}
		Let $K$ be a totally real number field of odd degree in which $2$ is inert. Let $l \geq 5$ be a prime with $l \nmid [K:\Q]$ and $\gcd(\frac{l-1}{2}, [K:\Q])=1$. Assume 
		\begin{enumerate}
			\item $l$ is non-Wieferich, i.e., $2^{l-1} \not\equiv 1 \pmod {l^2} $;
			\item $l$ is totally ramified in $K$.
		\end{enumerate}
		Then, for every $n\geq1$, the asymptotic Fermat's Last Theorem holds over the layer $K_{n,l}$ of the cyclotomic $\Z_l$-extension of $K$. Moreover, if all elliptic curves $E/K_{n,l}$ with full $2$-torsion are modular, then the effective asymptotic Fermat's Last Theorem holds over $K_{n,l}$.
	\end{thm}
	We now study asymptotic solutions of the generalized Fermat equation \eqref{Ax^p+By^p=Cz^p} 
	over the cyclotomic $\Z_l$-extension of $K$. 
	\begin{dfn}
		Let $K$ be a totally real field. A solution $(a, b, c)\in K^3$ to the equation~\eqref{Ax^p+By^p=Cz^p} is said to be non-trivial if $abc\neq 0$. We say a solution  $(a, b, c)\in \mcO_K^3$ is primitive if $a\mcO_K+b\mcO_K+c\mcO_K=\mcO_K$.
	\end{dfn}  
	\begin{dfn}
		\label{dfn for asymptotic solution}
		Let $K$ be a totally real field and let $S\subseteq K^3$. We say that the equation~\eqref{Ax^p+By^p=Cz^p} has no asymptotic solutions in $S$ if there exists a constant $V=V_{K,A,B,C}$ such that, for every prime $p>V$, the equation has no non-trivial solutions in $S$. If $V$ is effectively computable, then we say that the equation has no effective asymptotic solutions in $S$.
	\end{dfn}
	For any number field $F$, let $h_F^+$ denote its narrow class number. Assume $2$ is inert in $F$ and write $\mfP=2\mcO_F$. For any $A,B,C\in F^\times$, define
\begin{equation}
\label{coefficient nonvanishing quantity}
\begin{split}
\mcD_\mfP(A,B,C):={}&\bigl(8+v_\mfP(BCA^{-2})\bigr)\bigl(8+v_\mfP(ACB^{-2})\bigr)\\
&\mathrel{}\bigl(8+v_\mfP(ABC^{-2})\bigr).
\end{split}
\end{equation}
For each nonzero prime ideal $\mfq$ of $\mcO_F$, we use the normalized additive valuation $v_\mfq:F^\times\to\Z$, characterized by $v_\mfq(\pi)=1$ for any uniformizer $\pi$ at $\mfq$.
\begin{dfn}
For a number field $F$ and a prime $\mfP$ above $2$, define
\[
W_{F,\mfP}
:=\bigl\{(a,b,c)\in(F^\times)^3 \mid
v_\mfP(a),v_\mfP(b),v_\mfP(c)\text{ are not all equal}\bigr\}.
\]
\end{dfn}
For every $\xi\in F^\times$, the valuations of $\xi a,\xi b,\xi c$ are obtained by adding $v_\mfP(\xi)$ to those of $a,b,c$. Hence $(a,b,c)\in W_{F,\mfP}$ if and only if $(\xi a,\xi b,\xi c)\in W_{F,\mfP}$. Now let $(a,b,c)\in\mcO_F^3$ be a non-trivial primitive triple. Primitivity gives $\min\{v_\mfP(a),v_\mfP(b),v_\mfP(c)\}=0$, while integrality makes all three valuations nonnegative. Consequently,
\[
(a,b,c)\in W_{F,\mfP}
\iff \max\{v_\mfP(a),v_\mfP(b),v_\mfP(c)\}>0
\iff \mfP\mid abc.
\]

The proofs of the following three results are given in Section~\ref{sec extensions}.
We first consider the case when the coefficients $A,B,C$ are supported at the primes above $2$.
\begin{thm}
\label{main result3}
Let $K$ and $l$ satisfy the hypotheses of Theorem~\ref{main result2}, fix $n\geq1$, and put $\mfP=2\mcO_{K_{n,l}}$. Let
$
A,B,C\in\{u2^r:u\in\mcO_K^\times,\ r\in\Z_{\geq0}\}.
$
Then the following assertions hold.
\begin{enumerate}
\item Suppose that
$
\max\{v_\mfP(A),v_\mfP(BC)\}\leq4,$ and $
v_\mfP(ABC)\equiv0\text{ or }2\pmod3.
$
Then $Ax^p+By^p=Cz^p$ has no asymptotic solutions in $K_{n,l}^3$. In particular, for $r\in\{2,3\}$, the equation $x^p+y^p=2^rz^p$ has no asymptotic solutions in each layer $K_{n,l}^3$.
\item If $A=u_A2^a$, $B=u_B2^a$ and $C=u_C2^{a+k}$, where $u_A,u_B,u_C\in\mcO_K^\times$, $a\geq0$ and $k\geq5$, then $Ax^p+By^p=Cz^p$ has no asymptotic solutions in each layer $K_{n,l}^3$. In particular, for each fixed integer $r\geq5$, the equation $x^p+y^p=2^rz^p$ has no asymptotic solutions in each layer $K_{n,l}^3$.
\item If $\mcD_\mfP(A,B,C)\neq0$, then $Ax^p+By^p=Cz^p$ has no asymptotic solutions in $W_{K_{n,l},\mfP}$. In particular, this excludes every non-trivial primitive solution $(a,b,c)\in\mcO_{K_{n,l}}^3$ with $\mfP\mid abc$ for all sufficiently large primes $p$.
\end{enumerate}
\end{thm}

\begin{remark}\label{rem powers of two}
For $r=4$, the valuations $(0,0,4)$ fail the congruence in Theorem~\ref{main result3}(1) and the gap condition in part~(2), while $\mcD_\mfP(1,1,16)=0$ rules out part~(3). For $r=1$, the triple $(1,1,1)$ is a non-trivial solution for every prime $p$.
\end{remark}

We now consider the case when the coefficients $A,B,C$ are supported at the primes above $2d$, with $d$ an odd prime.
\begin{thm}
\label{thm extension two d}
Let $K$ and $l$ satisfy the hypotheses of Theorem~\ref{main result2}. Let $d\ne l$ be a prime inert in $K$ such that
$
d\equiv5\pmod8,\ d^{l-1}\not\equiv1\pmod{l^2}.
$
Fix $n\geq1$ such that $h_{K_{n,l}}^+$ is odd, and put $\mfP=2\mcO_{K_{n,l}}$. Let
$
A,B,C\in\{u2^rd^s:u\in\mcO_{K_{n,l}}^\times,\ r,s\in\Z_{\geq0}\}.
$
\begin{enumerate}
\item If $v_\mfP(A)=v_\mfP(B)=m$ and $v_\mfP(C)=m+k$, where $m\geq0$ and $k\geq5$, then $Ax^p+By^p=Cz^p$ has no asymptotic solutions in $K_{n,l}^3$.
\item If $\mcD_\mfP(A,B,C)\neq0$, then $Ax^p+By^p=Cz^p$ has no asymptotic solutions in $W_{K_{n,l},\mfP}$. In particular, this excludes every non-trivial primitive solution $(a,b,c)\in\mcO_{K_{n,l}}^3$ with $\mfP\mid abc$ for all sufficiently large primes $p$.
\end{enumerate}
\end{thm}
\begin{remark}\label{rem layer effectivity}
The bounds in Theorems~\ref{main result3} and~\ref{thm extension two d} are effectively computable if all elliptic curves with full rational $2$-torsion over $K_{n,l}$ are modular.
\end{remark}
\begin{cor}
\label{cor extension rational tower}
Let $l\geq5$ and $d\geq3$ be distinct primes satisfying
\[
2^{l-1}\not\equiv1\pmod{l^2}, \qquad d\equiv5\pmod8, \qquad d^{l-1}\not\equiv1\pmod{l^2}.
\]
Fix $n\geq1$ such that $h_{\Q_{n,l}}^+$ is odd, and put $\mfP=2\mcO_{\Q_{n,l}}$. For $A,B,C\in\{\pm2^rd^s:r,s\in\Z_{\geq0}\}$, the following assertions hold.
\begin{enumerate}
\item If $v_\mfP(A)=v_\mfP(B)=m$ and $v_\mfP(C)=m+k$ with $m\geq0$ and $k\geq5$, then $Ax^p+By^p=Cz^p$ has no effective asymptotic solutions in $\Q_{n,l}^3$.
In particular, for every fixed $r\geq5$ and $s\geq0$, the equation
$
x^p+y^p=2^rd^sz^p
$
has no effective asymptotic solutions in $\Q_{n,l}^3$.
\item If $\mcD_\mfP(A,B,C)\neq0$, then $Ax^p+By^p=Cz^p$ has no effective asymptotic solutions in $W_{\Q_{n,l},\mfP}$. 
\end{enumerate}
\end{cor}

\begin{example}\label{ex quintic layer}
Take $K=\Q$ and $l\in\{5,7,11,13\}$. Set $d=13$ for $l=5,7,11$, and $d=5$ for $l=13$. In each case $d\equiv5\pmod8$, while
\[
\begin{array}{c|cccc}
l&5&7&11&13\\ \hline
2^{l-1}\bmod l^2&16&15&56&40\\
d^{l-1}\bmod l^2&11&15&111&14
\end{array}
\]
shows that $2^{l-1},d^{l-1}\not\equiv1\pmod{l^2}$. By Lemma~\ref{lem for p is inert in Q_{n,l}}, both $2$ and $d$ are inert in $\Q_{n,l}$ for all $n\geq1$. Also, for all $n \geq 1$, the field $\Q_{n,l}$ has odd narrow class number~\cite[Lemma~3 and proof of Theorem~2]{HH08}. With $(A,B,C)=(1,1,2^5d)$, the valuations at $2\mcO_{\Q_{n,l}}$ are $(0,0,5)$. Thus, for each fixed $n\geq1$, Theorem~\ref{thm extension two d}(1) gives no asymptotic solutions of $x^p+y^p=2^5dz^p$ in $\Q_{n,l}^3$.
\end{example}

\begin{remark}
The conclusion of Theorem~\ref{main result3}(1) need not hold when the coefficients are supported above $2d$, as in Theorem~\ref{thm extension two d} and Corollary~\ref{cor extension rational tower}. For $d=5$ and $(A,B,C)=(1,4,5)$, we have $\max\{v_\mfP(A),v_\mfP(BC)\}= 2=v_\mfP(ABC)$. Yet $(1,1,1)$ solves $x^p+4y^p=5z^p$ for every prime $p$. Here $(\lambda,\mu)=(5,-4)$ is an $S'_{K_{n,l}}$-unit solution with $v_\mfP(\lambda\mu)=2\not\equiv1\pmod3$, so the unit congruence in Theorem~\ref{thm for S-unit crit even} fails.
\end{remark}

		\section{\texorpdfstring{The unit equation over cyclotomic $\Z_l$-extensions of $F$}{The unit equation over cyclotomic Z-l-extensions of F}}
		\label{preliminary section}
		Let $\zeta_m$ denote a primitive $m$-th root of unity. For an odd prime $l$ and $n\geq0$, let $\Q_{n,l}$ be the unique subfield of $\Q(\zeta_{l^{n+1}})$ of degree $l^n$ over $\Q$. For $l=2$, set
\[
\Q_{n,2}=\Q(\zeta_{2^{n+2}})^+,
\]
where the superscript $+$ denotes the maximal totally real subfield. Thus $\Q_{0,l}=\Q$ for every prime $l$, and $\Q_{1,2}=\Q(\sqrt2)$. The union
\[
\Q_{\infty,l}:=\bigcup_{i\geq0}\Q_{i,l}
\]
is the cyclotomic $\Z_l$-extension of $\Q$, and $\Q_{n,l}$ is its $n$-th layer.

For any number field $F$, set $F_{\infty,l}=F\Q_{\infty,l}$ and $F_{n,l}=F\Q_{n,l}$. If $F\cap\Q_{\infty,l}=\Q_{t,l}$, the degree-$l^n$ layer of $F_{\infty,l}/F$ is $F\Q_{n+t,l}$. In particular, $F_{n,l}$ is the $n$-th layer over $F$ when $F\cap\Q_{\infty,l}=\Q$; this holds whenever $l\nmid[F:\Q]$, as in the main results.

	The following unit-equation theorem is used in the proofs of the main results.
		\begin{thm}
\label{main result1 for unit equation}
Let $F$ be a number field of degree $m$, and let $l\geq5$ be a prime that is totally ramified in $F$. Assume $l\nmid m$ and $\gcd((l-1)/2,m)=1$. Then, for every $n\geq1$, the unit equation
\[
\lambda+\mu=1,\qquad \lambda,\mu\in\mcO_{F_{n,l}}^\times
\]
has no solutions.
\end{thm}
		\begin{proof}
Fix $n\geq1$ and put $M=F_{n,l}$. Since $\gcd(m,l^n)=1$, the fields $F$ and $\Q_{n,l}$ are linearly disjoint and $[M:\Q]=ml^n$. For any prime $\mfq$ of $M$ above $l$, its ramification index over $\Q$ is divisible by the ramification indices $m$ and $l^n$ in the two subfields. It is therefore divisible by $ml^n=[M:\Q]$, and must equal this degree. Hence $l$ is totally ramified in $M$, with a unique prime $\mfq$ of $M$ lying above $l$ and residue field $\mcO_M/\mfq\simeq\F_l$.

We first show that every unit $u\in\mcO_M^\times$ satisfies $u\equiv\pm1\pmod\mfq$. Choose $a\in\Z$ such that $u\equiv a\pmod\mfq$; then $l\nmid a$. Let $L$ be the normal closure of $M/\Q$. From $l\mcO_L=(\mfq\mcO_L)^{[M:\Q]}$ and unique factorization of ideals, we get $\sigma(\mfq\mcO_L)=\mfq\mcO_L$ for every $\sigma\in\Gal(L/\Q)$. Since $u-a\in\mfq\mcO_L$, it follows that $\sigma(u)-a=\sigma(u-a)\in\mfq\mcO_L$ for every $\sigma\in\Gal(L/\Q)$. Multiplying over the distinct embeddings $M\hookrightarrow L$ and using $(\mfq\mcO_L)\cap\Z=l\Z$, we obtain
\begin{equation}
\label{norm congruence at l}
N_{M/\Q}(u)\equiv a^{[M:\Q]}\pmod l.
\end{equation}
Since $u$ is a unit, $N_{M/\Q}(u)=\pm1$. Also, Euler's criterion gives $a^{(l-1)/2}\equiv\pm1\pmod l$. The assumption on $m$ implies $\gcd([M:\Q],(l-1)/2)=1$. Choose $b,c\in\Z$ with $b[M:\Q]+c(l-1)/2=1$. Since $a$ is invertible modulo $l$,
\[
a\equiv\bigl(a^{[M:\Q]}\bigr)^b\bigl(a^{(l-1)/2}\bigr)^c\equiv\pm1\pmod l.
\]
Consequently,
\begin{equation}
\label{congruence of norm}
u\equiv\pm1\pmod\mfq.
\end{equation}
If units $\lambda,\mu$ satisfied $\lambda+\mu=1$, reducing this equation modulo $\mfq$ would give $1\equiv\pm1\pm1\pmod l$. This is impossible for $l\geq5$.
\end{proof}
		
		\section{\texorpdfstring{Generalized Fermat equation over cyclotomic $\Z_l$-extensions of $K$}{Generalized Fermat equation over cyclotomic Z-l-extensions of K}}
		\label{preliminary for the GFE}
		Let $K$ be a totally real number field, and let $P_K$ denote the set of all non-zero prime ideals of $\mcO_K$. Put $S_K:= \{ \mfP \in P_K :\ \mfP|2\}$. For any $A,B,C \in \mcO_K\setminus \{0\}$, let $S_K^{\prime}:= \{ \mfP \in P_K :\ \mfP|2ABC \}$.  For any set $S \subseteq P_K$, let $\mcO_{S}:=\{\alpha \in K : v_\mfP(\alpha)\geq 0 \text{ for all } \mfP \in P_K \setminus S\}$ denote the ring of $S$-integers of $K$ and let $\mcO_{S}^\times$ denote the group of $S$-units of $K$. 
		\subsection{The $S$-unit bound at the prime above two and the proof of Theorem~\ref{main result2}}
		The following criterion is a specialization of \cite[Theorem 2.7]{S26}. It will be used in the proofs of Theorems~\ref{main result2} and~\ref{main result3}.
		\begin{thm}
			\label{thm for S-unit crit even}
			Let $K$ be a totally real number field of odd degree in which $2$ is inert and write $\mfP=2\mcO_K$. Let $A,B,C \in \mcO_K\setminus \{0\}$. Suppose every solution $(\lambda, \mu)$ to the $S_K^\prime$-unit equation
			\begin{equation}
				\label{S_K-unit solution}
				\lambda+\mu=1, \ \lambda, \mu \in \mcO_{S_K^\prime}^\times
			\end{equation}
			satisfies
			\begin{equation}
				\label{assumption for main result x^p+y^p=2^rz^p}
				\max \left\{|v_\mfP(\lambda)|,|v_\mfP(\mu)| \right\}\leq 4, \text{ and }  v_\mfP(\lambda\mu)\equiv 1\pmod 3.
			\end{equation}
				If $\max\{v_\mfP(A), v_\mfP(BC)\} \leq 4$ and  $v_\mfP(ABC) \equiv 0$ or $2 \pmod 3$, then the generalized Fermat equation $Ax^p+By^p=Cz^p$ has no asymptotic solution in $K^3$. Moreover, if all elliptic curves $E/K$ with full $2$-torsion are modular, then the equation has no effective asymptotic solutions in $K^3$.
		\end{thm}
\begin{remark}
The condition $A\pm B\pm C\neq0$ in \cite[Theorem 2.7]{S26} is automatic. Indeed, if $A\pm B\pm C=0$ for some choice of signs, then there are $S_K'$-units $\lambda=\pm A/C$ and $\mu=\pm B/C$ with $\lambda+\mu=1$. Hence
$
1\equiv v_\mfP(\lambda\mu)=v_\mfP(ABC)-3v_\mfP(C)\equiv v_\mfP(ABC)\pmod3,
$
which is a contradiction to $v_\mfP(ABC)\equiv0$ or $2\pmod3$. 
\end{remark}

		The following lemma is a key ingredient in the proof of Theorems~\ref{main result2} and~\ref{main result3}.
		\begin{lem}
			\label{lem for solution of S-unit eqn}
			Fix $n\geq1$. Let $K$ be a totally real number field of odd degree. Let $l \geq 5$ be a prime that is totally ramified in $K$ with $l \nmid [K:\Q]$ and $\gcd(\frac{l-1}{2}, [K:\Q])=1$. Assume $2$ is inert in $F= K_{n,l}$. Write $\mfP=2\mcO_F$ and $S_F=\{\mfP\}$. Then every solution $(\lambda, \mu)$ to the $S_F$-unit equation
			$\lambda+\mu=1, \ \lambda, \mu \in \mcO_{S_F}^\times$ satisfies
			$ \left(v_\mfP(\lambda), v_\mfP(\mu)\right) \in \{(1,0), (0,1), (-1,-1)\}$.
		\end{lem}
		\begin{proof}
			We first show that $v_\mfP(\lambda) <2$. Suppose instead that $v_\mfP(\lambda) \geq 2$. This gives $\lambda \in \mcO_F$ and $\lambda \equiv 0 \pmod {\mfP^2}$. Since $\lambda+\mu=1$, we get $v_\mfP(\mu)=0$ and $\mu \equiv 1 \pmod {\mfP^2}$. This gives $\mu \in \mcO_F^\times$ and $\text{N}_{F/\Q}(\mu) \equiv 1 \pmod 4$, hence $\text{N}_{F/\Q}(\mu)=1$. 
			The coprime-degree ramification argument in the proof of Theorem~\ref{main result1 for unit equation} shows that $l$ is totally ramified in $F=K_{n,l}$. Let $\mfq$ be the unique prime of $F$ lying above $l$.  From \eqref{congruence of norm}, it follows that $\mu \equiv \pm 1 \pmod \mfq$. If $\mu \equiv 1 \pmod \mfq$, then  $\lambda \equiv 0 \pmod \mfq$, which is not possible since $\lambda \in \mcO_{S_F}^\times$ and $\mfq  \notin S_F$. Hence $\mu \equiv -1 \pmod \mfq$. Applying \eqref{norm congruence at l} with $M=F$, $u=\mu$ and $a=-1$ gives 
			$1=\text{N}_{F/\Q}(\mu) \equiv (-1)^{[F: \Q] }\pmod l$. Since $[K : \Q]$ is odd and $l$ is odd, it follows that $[F: \Q] $ is odd. Thus, we have $1 \equiv -1\pmod l$, which is not possible since $l \geq 5$.
			
			Next, we show that $v_\mfP(\lambda) >-2$. Suppose instead that $v_\mfP(\lambda) \leq -2$. This gives $v_\mfP(\lambda)=v_\mfP(\mu)$. Choose $\lambda'=\frac{1}{\lambda}$ and $\mu'= \frac{-\mu}{\lambda}$. Therefore,  $\lambda', \mu' \in \mcO_{S_F}^\times$, $\lambda'+\mu'=1$ and $v_\mfP(\lambda') \geq 2$, which contradicts the previous case. Hence, we have $-2 < v_\mfP(\lambda) <2$ and by symmetry  $-2 <v_\mfP(\mu) <2$. Finally, if $v_\mfP(\lambda) =0=v_\mfP(\mu) $, then $\lambda, \mu \in \mcO_F^\times$, which is not possible by Theorem~\ref{main result1 for unit equation}. This finishes the proof of the lemma.
		\end{proof}	
		We now recall a result of \cite{FKS20a} that provides a necessary and sufficient condition for a prime to be inert in $\Q_{n,l}$.
		\begin{lem}[{\cite[Lemma 2.1]{FKS20a}}]
			\label{lem for p is inert in Q_{n,l}}
			Fix $n\geq1$, and let $l \geq 3$ and $d\geq 2$ be distinct rational primes. Then $d$ is inert in $\Q_{n,l}$ if and only if $d^{l-1} \not\equiv 1 \pmod {l^2} $.
		\end{lem}
		We now prove Theorem~\ref{main result2}.
		\begin{proof}[Proof of Theorem~\ref{main result2}]
Put $F=K_{n,l}$ and $m=[K:\Q]$. Since $\gcd(m,l^n)=1$, we have $[F:\Q]=ml^n$. By Lemma~\ref{lem for p is inert in Q_{n,l}}, $2$ is inert in $\Q_{n,l}$. The completions of $K$ and $\Q_{n,l}$ at their primes above $2$ are unramified extensions of $\Q_2$ of degrees $m$ and $l^n$. Their compositum has degree $\operatorname{lcm}(m,l^n)=ml^n$. This equals $[F:\Q]$, so $2$ is inert in $F$. Write $\mfP=2\mcO_F$ and $S_F=\{\mfP\}$.

Lemma~\ref{lem for solution of S-unit eqn} gives
\[
(v_\mfP(\lambda),v_\mfP(\mu))\in\{(1,0),(0,1),(-1,-1)\}
\]
for every $S_F$-unit solution. Hence the maximum of the absolute valuations is $1$, and $v_\mfP(\lambda\mu)\equiv1\pmod3$. The field $F$ is totally real of odd degree. Take $A=B=C=1$. Then $S_F'=S_F$, and $v_\mfP(A)=v_\mfP(BC)=v_\mfP(ABC)=0$, so both coefficient conditions of Theorem~\ref{thm for S-unit crit even} hold. Hence $x^p+y^p=z^p$ has no asymptotic solutions in $F^3$, and the bound is effective if all elliptic curves over $F$ with full rational $2$-torsion are modular. This completes the proof of the theorem.
\end{proof}

	\subsection{A $S$-unit bound for coefficients supported above two primes}
		The following proposition gives a valuation bound for solutions to the $T$-unit equation, where $T$ is the set of primes lying above $2d$.
\begin{prop}
\label{prop for solution of S-unit eqn even soln}
\label{Lemma for S unit crit 2d}
Let $K$ be a number field in which $2$ is inert, and write $\mfP=2\mcO_K$. Let $d\geq3$ be a prime and put $T=\{\mfq\in P_K:\mfq\mid2d\}$. Assume
\begin{enumerate}
\item $2\nmid h_K^+$;
\item $d\equiv1\pmod4$ and $d$ is inert in $K$;
\item the congruence $d\equiv v^2\pmod{\mfP^5}$ has no solution $v\in\mcO_K$.
\end{enumerate}
Then every solution of
\[
\lambda+\mu=1,\qquad \lambda,\mu\in\mcO_T^\times
\]
satisfies $\max\{|v_\mfP(\lambda)|,|v_\mfP(\mu)|\}\leq4$.
\end{prop}

	We use the following elementary normalization, also used in \cite[Lemma 5.2]{JS25}.
		\begin{lem}
\label{integrality of lambda, mu}
Let $K$ be a number field, let $T$ be a finite set of nonzero prime ideals, and let $\mfP\in T$. For any solution $\lambda+\mu=1$ with $\lambda,\mu\in\mcO_T^\times$, there is a solution $(\lambda',\mu')$ with $v_\mfP(\lambda'),v_\mfP(\mu')\geq0$ and
$
\max\{|v_\mfP(\lambda)|,|v_\mfP(\mu)|\}=\max\{v_\mfP(\lambda'),v_\mfP(\mu')\}.
$
\end{lem}
\begin{proof}
If $v_\mfP(\lambda),v_\mfP(\mu)\geq0$, then take $(\lambda',\mu')=(\lambda,\mu)$.  Otherwise $v_\mfP(\lambda)=v_\mfP(\mu)<0$, and the pair $(\lambda',\mu')=(1/\lambda,-\mu/\lambda)$ has valuations $-v_\mfP(\lambda)>0$ and $0$. It is a $T$-unit solution and has the required maximum.
\end{proof}
		\begin{proof}[Proof of Proposition~\ref{prop for solution of S-unit eqn even soln}]
We follow the descent used in \cite[Theorem 6]{M23} and \cite[Proposition 1.9]{S25}. Suppose a $T$-unit solution has maximum absolute valuation at least $5$. By Lemma~\ref{integrality of lambda, mu} and interchanging $\lambda,\mu$ if necessary, we get a solution $(\lambda,\mu)$ to the $T$-unit equation with $v_\mfP(\lambda)\geq5$ and $v_\mfP(\mu)=0$. The $T$-unit equation has finitely many solutions, so choose $(\lambda',\mu')$ with $s_0=v_\mfP(\lambda')$ maximal. Then $s_0\geq5$ and $\mu'\equiv1\pmod{\mfP^5}$.

We show that $\mu'$ is a square in $\mcO_T^\times$. Since $d$ is inert, $d\mcO_K$ is a prime ideal, and the principal ideal of $\mu'$ is supported only at this prime. Therefore
\begin{equation}
\label{factor of mu}
\mu'=\alpha d^s,\qquad \alpha\in\mcO_K^\times,\quad s\in\Z.
\end{equation}
Since $d$ is odd, it is invertible modulo $\mfP^5$, so $d^s$ is defined modulo $\mfP^5$ even when $s<0$. As $\mfP=2\mcO_K$ and $\mu'\equiv1\pmod{\mfP^5}$, we have $\mu'\equiv1\pmod{4\mcO_K}$. Moreover, $d\equiv1\pmod4$ implies $d^s\equiv1\pmod{4\mcO_K}$ for every $s\in\Z$. Therefore
\[
\alpha=\mu'd^{-s}\equiv1\pmod{4\mcO_K}.
\]
If $\alpha$ is not a square in $K$, then $L=K(\sqrt\alpha)$ is a quadratic extension generated by the integral element $\beta=(1-\sqrt\alpha)/2$, whose minimal polynomial is
\[
X^2-X+\frac{1-\alpha}{4}.
\]
Its discriminant $\alpha$ is a unit. Hence $L/K$ is unramified at every finite place. Such a quadratic extension is excluded by $2\nmid h_K^+$, whether or not it ramifies at real places. Thus $\alpha=\delta^2$ for a unit $\delta\in\mcO_K^\times$.
If $s=2k+1$ were odd, put $t=\delta d^k$. Since $\delta$ is a unit and $d$ is odd, $t$ is invertible modulo $\mfP^5$. Choose $v\in\mcO_K$ such that $tv\equiv1\pmod{\mfP^5}$. Then $\mu'=t^2d\equiv1\pmod{\mfP^5}$ implies
\[
d\equiv v^2\pmod{\mfP^5},
\]
which contradicts the hypothesis (3). Hence $s$ is even and $\mu'=\gamma^2$ for some $\gamma\in\mcO_T^\times$ with $v_\mfP(\gamma)=0$.

Now $\lambda'=1-\mu'=(1+\gamma)(1-\gamma)$. Since $\gamma\in\mcO_T^\times$, for every prime $\mfq\notin T$, we have $v_\mfq(1+\gamma),\ v_\mfq(1-\gamma)\geq0$ and 
$
v_\mfq(1+\gamma)+v_\mfq(1-\gamma)=v_\mfq(\lambda')=0.
$
Hence $v_\mfq(1+\gamma)=0=v_\mfq(1-\gamma)$, therefore $1+\gamma, \ 1-\gamma\in\mcO_T^\times$. Put $s_1=v_\mfP(1+\gamma)$ and $s_2=v_\mfP(1-\gamma)$. Since $v_\mfP(\gamma)=0$, it follows that $s_1, s_2 \geq 0$ and $s_1+s_2=s_0\geq5$. Since $(1+\gamma)+(1-\gamma)=2$ has valuation $1$, it follows that $\min(s_1,s_2)=1$. Replacing $\gamma$ by $-\gamma$ if necessary, take $s_1=1$ and $s_2=s_0-1$.
The pair
\[
\lambda''=-\frac{(1-\gamma)^2}{4\gamma},\qquad
\mu''=\frac{(1+\gamma)^2}{4\gamma}
\]
is a $T$-unit solution of $\lambda''+\mu''=1$, but
\[
v_\mfP(\lambda'')=2s_2-2=2s_0-4>s_0.
\]
This contradicts maximality of $s_0$ and hence the proof follows.
\end{proof}

	The following norm congruence will be used in the proof of Theorem~\ref{thm extension two d}.
		\begin{lem}
\label{Norm red lemma}
Let $F$ be a number field and $n\geq1$. If $\lambda,\nu\in\mcO_F$ satisfy $\lambda\equiv\nu\pmod{2^n\mcO_F}$, then
$
N_{F/\Q}(\lambda)\equiv N_{F/\Q}(\nu)\pmod{2^n}.
$
\end{lem}
		\begin{proof}
Let $L$ be a normal closure of $F/\Q$. For every embedding $\sigma:F\hookrightarrow L$, the assumption $\lambda-\nu\in2^n\mcO_F$ gives $\sigma(\lambda)-\sigma(\nu)\in2^n\mcO_L$. Multiplying these congruences over all embeddings yields
$
N_{F/\Q}(\lambda)\equiv N_{F/\Q}(\nu)
\pmod{2^n\mcO_L}.
$
Both norms are rational integers, and $2^n\mcO_L\cap\Z=2^n\Z$. Hence they are congruent modulo $2^n$.
\end{proof}
\subsection{Further coefficient ranges and proofs of Theorems~\ref{main result3} and~\ref{thm extension two d}}
\label{sec extensions}
We first record the specialization of \cite[Theorem 1.6]{S26new} used below.
\begin{thm}[{\cite[Theorem 1.6]{S26new}}]
\label{thm new criterion}
Let $F$ be a totally real number field in which $2$ is inert, put $\mfP=2\mcO_F$, and let $A,B,C\in\mcO_F\setminus\{0\}$. Set $S'=\{\mfq:\mfq\mid2ABC\}$. Suppose every solution of
\[
\lambda+\mu=1,\qquad \lambda,\mu\in\mcO_{S'}^\times
\]
satisfies $\max\{|v_\mfP(\lambda)|,|v_\mfP(\mu)|\}\leq4$.
\begin{enumerate}
\item If $v_\mfP(A)=v_\mfP(B)=m$ and $v_\mfP(C)=m+k$ with $m\geq0$ and $k\geq5$, then $Ax^p+By^p=Cz^p$ has no asymptotic solutions in $F^3$.
\item If $\mcD_\mfP(A,B,C)\neq0$, then $Ax^p+By^p=Cz^p$ has no asymptotic solutions in $W_{F,\mfP}$. In particular, this excludes every non-trivial primitive integral solution $(a,b,c)$ with $\mfP\mid abc$ for all sufficiently large primes $p$.
\end{enumerate}
\end{thm}
\begin{proof}
Apply the respective part of \cite[Theorem 1.6]{S26new} to $(A,B,C)$. Since $2$ is inert, $v_\mfP(2)=1$, so the required $S'$-unit bound is exactly $4$, and the strict inequality $k>4v_\mfP(2)$ is equivalent to $k\geq5$. The product $\alpha_\mfP\beta_\mfP\gamma_\mfP$ in that theorem equals $\mcD_\mfP(A,B,C)$. This completes the proof.
\end{proof}

\begin{proof}[Proof of Theorem~\ref{main result3}]
As shown in the proof of Theorem~\ref{main result2}, $F=K_{n,l}$ is totally real of odd degree and $2$ is inert in $F$. Hence $\mfP=2\mcO_F$ is the unique prime above $2$ and $v_\mfP(2)=1$. Every unit of $K$ remains a unit in $F$, so $S_F'=\{\mfq:\mfq\mid2ABC\}=\{\mfP\}$. By Lemma~\ref{lem for solution of S-unit eqn}, every solution $(\lambda,\mu)$ of the $S_F'$-unit equation satisfies $\bigl(v_\mfP(\lambda),v_\mfP(\mu)\bigr)\in\{(1,0),(0,1),(-1,-1)\}$. In each case, $\max\{|v_\mfP(\lambda)|,|v_\mfP(\mu)|\}=1$ and $v_\mfP(\lambda\mu)\equiv1\pmod3$.

For part~(1), every $S_F'$-unit solution satisfies $\max\{|v_\mfP(\lambda)|,|v_\mfP(\mu)|\}=1 <4$ and $v_\mfP(\lambda\mu)\equiv1\pmod3$. Hence Theorem~\ref{thm for S-unit crit even} gives the claimed asymptotic nonexistence. If $(A,B,C)=(1,1,2^r)$, then $\max\{v_\mfP(A),v_\mfP(BC)\}=v_\mfP(ABC)=r$. For $r=2,3$, the required bound and congruence both hold.

For part (2), the hypothesis on $A,B,C$ gives $(v_\mfP(A), v_\mfP(B), v_\mfP(C)) =(a,a,a+k)$. Since $k\geq5$ and $\max\{|v_\mfP(\lambda)|,|v_\mfP(\mu)|\}=1$, Theorem~\ref{thm new criterion}(1) applies. Taking $(A,B,C)=(1,1,2^r)$ gives the stated consequence for $r\geq5$. For part (3), since  $\max\{|v_\mfP(\lambda)|,|v_\mfP(\mu)|\}=1$ and $\mcD_\mfP(A,B,C)\neq0$, the conclusion follows by Theorem~\ref{thm new criterion}(2). Hence the proof follows.
\end{proof}

\begin{proof}[Proof of Theorem~\ref{thm extension two d}]
We first obtain the required $T$-unit bound for any totally real field $F$ of odd narrow class number in which $2$ and $d$ are inert. Put $\mfP_F=2\mcO_F$ and let $T$ be the set of primes above $2d$. Assume $d\equiv1\pmod4$ and
$
d^{[F:\Q]}\not\equiv1,9,17,25\pmod{32}.
$
If $d\equiv v^2\pmod{\mfP_F^5}$ for some $v\in\mcO_F$, then $\mfP_F^5=32\mcO_F$ and Lemma~\ref{Norm red lemma} yields
$
d^{[F:\Q]}\equiv N_{F/\Q}(v)^2\pmod{32}.
$
Since $d$ is odd, $N_{F/\Q}(v)$ is odd modulo $2$; its square is therefore one of $1,9,17,25$ modulo $32$, a contradiction. Thus $d$ is not a square modulo $\mfP_F^5$, and Proposition~\ref{prop for solution of S-unit eqn even soln} gives
\[
\max\{|v_{\mfP_F}(\lambda)|,|v_{\mfP_F}(\mu)|\}\leq4
\]
for every solution $\lambda+\mu=1$ with $\lambda,\mu\in\mcO_T^\times$. If $A,B,C$ are supported at the primes above $2d$, then $S'=\{\mfq:\mfq\mid2ABC\}\subseteq T$, so the same bound holds for every $S'$-unit solution.

Take $F=K_{n,l}$ and put $N=[K:\Q]$. By the proof of Theorem~\ref{main result2}, $F$ is totally real of odd degree $Nl^n$ and $2$ is inert in $F$. Lemma~\ref{lem for p is inert in Q_{n,l}} gives the inertness of $d$ in $\Q_{n,l}$. Since $d$ is also inert in $K$ and $\gcd(N,l^n)=1$, the corresponding unramified completions have compositum of degree $Nl^n=[F:\Q]$. Thus $d$ is inert in $F$. 

As $d\equiv5\pmod8$ and $[F:\Q]$ is odd, we have
$
d^{[F:\Q]}\equiv5\pmod8.
$
Since $1,9,17,25 \equiv  1\pmod8$, it follows that
$
d^{[F:\Q]}\not\equiv1,9,17,25\pmod{32}.
$
Also $d\equiv1\pmod4$ and $h_F^+$ is odd. The first paragraph therefore shows that $d$ is not a square modulo $\mfP_F^5$ and that every $T$-unit solution $\lambda+\mu=1$ satisfies
\[
\max\{|v_{\mfP_F}(\lambda)|,|v_{\mfP_F}(\mu)|\}\leq4.
\]
For the coefficients $A,B,C$ in the theorem, $S'=\{\mfq:\mfq\mid2ABC\}\subseteq T$, so this bound holds for every $S'$-unit solution. The two assertions now follow from Theorem~\ref{thm new criterion}(1) and (2), respectively. This completes the proof of the theorem.
\end{proof}

\begin{proof}[Proof of Corollary~\ref{cor extension rational tower}]
By Lemma~\ref{lem for p is inert in Q_{n,l}}, the conditions $2^{l-1}\not\equiv1\pmod{l^2}$ and $d^{l-1}\not\equiv1\pmod{l^2}$ ensure that both $2$ and $d$ are inert in $\Q_{n,l}$. Since $l^n$ is odd, $d^{l^n}\equiv5\pmod8$. By Lemma~\ref{Norm red lemma}, $d$ cannot be a square modulo $\mfP^5$. Theorem~\ref{thm extension two d} now applies with $K=\Q$. For $(A,B,C)=(1,1,2^rd^s)$, the valuations at $\mfP$ are $(0,0,r)$. Since $r \geq 5$, part~(1) follows. Effectivity over $\Q_{n,l}$ now follows from~\cite[Theorem~1]{Thorne19}.
\end{proof}


\end{document}